\documentclass[12pt]{amsart}

\usepackage{amsmath, amssymb, bm, amsrefs}
\usepackage{hyperref}
\usepackage[all]{xy}
\usepackage{diagrams}
\usepackage{enumerate}

\newtheorem{theorem}{Theorem}[section]
\newtheorem{proposition}[theorem]{Proposition}
\newtheorem{lemma}[theorem]{Lemma}
\newtheorem{corollary}[theorem]{Corollary}
\newtheorem{conjecture}[theorem]{Conjecture}

\theoremstyle{definition}
\newtheorem{remark}[theorem]{Remark}
\newtheorem{example}[theorem]{Example}
\newtheorem{dfn}[theorem]{Definition}

\title [Standard Norm Varieties for Milnor Symbols mod $p$]{Standard Norm Varieties for Milnor Symbols mod $p$}
\author [Dinh Huu Nguyen] {Dinh Huu Nguyen}
\email {dhn {\it at} ucla.edu}

\begin{document}
\maketitle
Abstract: We prove that the standard norm varieties for Milnor symbols mod $p$ of length $n$ are birationally isomorphic to Pfister quadrics when $p = 2$, to Severi-Brauer varieties when $p > 2, n = 2$, and to varieties defined by reduced norms of cyclic algebras when $p > 2, n = 3$. In the case $p = 2$ and the case $p > 2, n = 2$, the results imply that the standard norm varieties for two equal Milnor symbols mod $p$ are birationally isomorphic, and we conjecture this in general.

\section{Introduction}
The norm residue theorem relates the Milnor K-theory mod $p$ of a field $k$ with the \'{e}tale cohomology of $k$ with coefficients in the twists of $\mu_p$. More precisely, it states that for each prime $p \neq \text{char}(k)$ and each weight $n \geq 0$ there exists an isomorphism
$$K_n^M(k)/p \cong H_{\acute{e}t}^n(k, \mu_p^n)$$

In 1996, V. Voevodsky proved the special case of $p = 2$, known as the Milnor conjecture in \cite{VoevodskyMCwithZ/2 -c}. In 2011, he proved the general case of the norm residue theorem, also known as the Bloch-Kato conjecture in \cite{VoevodskyOnMCwithZ/l-c}. His proof used a splitting variety with certain properties for a given Milnor symbol $\{a_1,..., a_n\}$ in $K^M_n(k)/p$. One construction for such splitting varieties was provided by M. Rost in \cite[Section 3]{HaesemeyerWeibelNVandtheCL}. Another construction for these varieties was suggested by V. Voevodsky in \cite[Section 2]{JoukhovitskiSuslinNV}. The entire theorem is being written in book form by C. Haesemeyer and C. Weibel in \cite{HaesemeyerWeibeltheNRTinMC}.

In section \ref{symmetricpowers}, we summarize Voevodsky's construction. It uses symmetric powers and produces what are called standard norm varieties.

In section \ref{p=2,alln}, we show in theorem \ref{varietytoquadric} that the standard norm varieties are birationally isomorphic to Pfister quadrics defined by subforms of Pfister forms when $p = 2$. Then we combine this result with the Chain P-equivalence Theorem by R. Elman and T. Y. Lam \cite[Main Theorem 3.2]{ELPFandtheKofF} and properties of quadratic forms to prove that the standard norm varieties for two equal symbols are birationally isomorphic in corollary \ref{p=2corollary}.

In section \ref{p>2,n=2}, we use Galois descent to show in theorem \ref{varietytoSB} that the standard norm varieties are birationally isomorphic to Severi-Brauer varieties when $p > 2, n = 2$ and get a similar corollary \ref{n=2corollary}.

In section \ref{p>2,n=3}, we use Galois descent to show in theorem \ref{varietytovarietydefinedbyreducednormofcyclicalgebra} that the standard norm varieties are birationally isomorphic to varieties defined by reduced norms of cyclic algebras when $p > 2, n = 3$. N. Karpenko and A. Merkurjev use this result and induction to prove $A$-triviality for standard norm varieties in \cite{KarpenkoMerkurjevOSNV}.

Given the above two corollaries, we make the following conjecture.
\begin{conjecture} The standard norm varieties for $\{a_1,..., a_n\}$ and $\{b_1,..., b_n\}$ are birationally isomorphic if $\{a_1,..., a_n\} = \{b_1,..., b_n\}$ in $K^M_n(k)/p$ for all $p, n$.
\end{conjecture}

Acknowledgements: This paper was my thesis, I thank my advisor Alexander Merkurjev for his guidance and my colleague Aaron Silberstein for his influence on submitting it for publication.

\section{Symmetric Powers}\label{symmetricpowers}
Throughout this paper, $p$ is a prime and $k$ is a base field of characteristic $0$ containing the $p^{\text{th}}$ roots of unity. Associated to each nontrivial Milnor symbol $\{a_1,..., a_n\}$ in $K^M_n(k)/p$ are the following notions. A general reference for Milnor K-theory is \cite{MilnorAK-TandQF}.

\dfn A field extension $L/k$ is called a {\it splitting field} for $\{a_1,..., a_n\}$ if $\{a_1,..., a_n\} = 0$ in $K^M_n(L)/p$.

\dfn A smooth variety $X$ is called a {\it splitting variety} for $\{a_1,..., a_n\}$ if its function field $k(X)$ is a splitting field for $\{a_1,..., a_n\}$. In addition, it is called a {\it generic splitting variety} for $\{a_1,..., a_n\}$ if any splitting field $L$ for $\{a_1,..., a_n\}$ has a point in $X$, i.e. if there exists a morphism $Spec(L) \rightarrow X$ over $k$.

Such generic splitting varieties are known to exist for all $n$ when $p = 2$ and only for $n \leq 3$ when $p > 2$. However, if $L'/L$ if a finite extension of degree prime to $p$ and $L'$ splits $\{a_1,..., a_n\}$ then $L$ also splits $\{a_1,..., a_n\}$ (using transfer and norm maps). Therefore we can relax our last definition.

\dfn A smooth variety $X$ is called a $p${\it -generic splitting variety} for $\{a_1,..., a_n\}$ if it is a splitting variety for $\{a_1,..., a_n\}$ and for any splitting field $L$ for $\{a_1,..., a_n\}$ there exists an extension $L'/L$ of degree prime to $p$ with a point in $X$. In addition, it is called a {\it norm variety} for $\{a_1,..., a_n\}$ if it is projective and geometrically irreducible of dimension $p^{n-1} - 1$.

\begin{example}\label{} When $n = 1$, a norm variety for $\{a_1\}$ is $Spec(L)$ where $L = k(\sqrt[p]{a_1})$. When $n = 2$, a norm variety for $\{a_1, a_2\}$ is the Severi-Brauer variety $SB(A)$ associated to the cyclic algebra $A = (a_1, a_2, \zeta_p)_k$. 
\end{example}

We now describe a standard way to produce these norm varieties for all $n$, which are called {\it standard norm varieties}.

Let $X$ be a smooth quasi-projective geometrically irreducible variety. The symmetric group $S_p$ acts on the product $X^p$ and induces the quotient variety $S^p(X)$. This quotient variety is geometrically irreducible and normal. Note that $S_p$ acts freely on $X^p \backslash \bigtriangleup$ and $U := (X^p \backslash \bigtriangleup)/S_p$ is an open subset in $S^p(X)$, where $\bigtriangleup$ is the union of all diagonals in $X^p$.

For every normal and irreducible scheme $Y$, the set of morphisms $Hom(Y, S^p(X))$ can be identified with the set of all effective cycles $Z \subset X \times Y$ such that each component of $Z$ is finite surjective over $Y$ and that the degree of $Z$ over $Y$ is $p$. In particular, the identity map $S^p(X) \rTo^{id} S^p(X)$ corresponds to the incidence cycle $Z \subset X \times S^p(X)$. In fact $Z$ is a closed subscheme, it is the image of the closed embedding $X \times S^{p-1}(X) \rTo X \times S^p(X), (x, y) \mapsto (x, x + y)$. Compose this with projection onto the second factor and we get a map $p: X \times S^{p-1}(X) \rightarrow X \times S^p(X) \rightarrow S^p(X)$. It is finite surjective of degree $p$. Thus we get a diagram
\begin{diagram}
X \times S^{p-1}(X) & \lTo & p^{-1}(U) & \lTo^{/S_{p-1}} & X^p \backslash \bigtriangleup \\ 
\dTo^p & & \dTo^{p|_{p^{-1}(U)}} & \ldTo^{/S_p} & \\
S^p(X) & \lTo & U & & \\
\end{diagram}

We see that both maps from $X^p \backslash \bigtriangleup$ are Galois \'{e}tale coverings, $p|_{p^{-1}(U)}$ is a finite \'{e}tale map of degree $p$, and $U$ is smooth. Furthermore $p_*f(\mathcal{O}_{X \times S^{p-1}(X)})$ is a coherent $\mathcal{O}_{S^p(X)}$-algebra and the sheaf $ \mathcal{A} := p_*(\mathcal{O}_{X \times S^{p-1}(X)} |_{p^{-1}(U)})$ is a locally free $\mathcal{O}_U$-algebra of rank $p$. This latter sheaf corresponds to the vector bundle $V:= Spec(S^{\bullet}(\mathcal{A}^{\nu}))$ of rank $p$ over $U$. Here $\mathcal{A}^{\nu}$ denotes the dual of $\mathcal{A}$ and $S^{\bullet}(\mathcal{A}^{\nu})$ denotes its symmetric algebra. There is a well defined norm map $\mathcal{A} \rTo^{N} \mathcal{O}_U$. Locally $N$ is a homogeneous polynomial of degree $p$, that is, $N \in S^p(\mathcal{A}^{\nu})$.

A norm variety $X(a_1,..., a_n)$ for $\{a_1,..., a_n\}$ is then constructed by induction. For $n = 2$, we take $X = X(a_1, a_2)$ in the preceding construction to be the Severi-Brauer variety $SB(A)$ associated to the cyclic algebra $A = (a_1, a_2, \zeta_p)_k$. Suppose we have constructed a norm variety $X(a_1,..., a_{n-1})$ for $\{a_1,..., a_{n-1}\}$. Again let that be $X$ and let $W \subset V$ be the hypersurface defined by the equation $N - a_n = 0$. By construction $W$ has dimension $p^{n-1} - 1$. By \cite[2.1]{JoukhovitskiSuslinNV} it is smooth over $U$ (hence smooth) and geometrically irreducible. By resolution of singularities we can embed $W$ as an open subvariety of a new smooth projective geometrically irreducible variety $X'$ of the same dimension. Together \cite[2.3]{JoukhovitskiSuslinNV}, \cite[2.4]{JoukhovitskiSuslinNV} and its subsequent argument show this $X'$ is a $p$-generic splitting variety for $\{a_1,..., a_n\}$. Hence $X'$ is the norm variety that we seek. Note that its construction depends solely on the tuple $(a_1,..., a_n)$.

\begin{remark} \label{startofconstruction} The inductive construction could in fact start with $n = 1$. We describe explicitly what happens at this stage. Take $X = X(a_1) = Spec(L)$ where $L = k(\sqrt[p]{a_1})$. If $\bar{k}$ is the separable closure of $k$ then $\overline{X} = X \times_{Spec(k)} Spec(\bar{k})$ has $p$ points, call them $1, 2,..., p - 1, p$. From there,

$\overline{X}^p = \{\text{points on the diagonals}\} \sqcup \{(n_1, n_2,..., n_p) \mid 1 \leq n_i \leq p \text{ and } n_i \neq n_j \text{ for all } i, j\}$

$S^p(\overline{X}) = \overline{X}^p/S_p = \{\text{classes of points on the diagonals}\} \sqcup \{\overline{(1, 2,..., p)}\}$

$\overline{X}^p \backslash \bigtriangleup = \{(n_1, n_2,..., n_p) \mid 1 \leq n_i \leq p \text{ and } n_i \neq n_j \text{ for all } i, j\}$

$(\overline{X}^p \backslash \bigtriangleup)/S_p = \{\overline{(1, 2,..., p)}\}$
\newline 

The above square thus looks like this
\begin{diagram}
\overline{X} \times S^{p-1}(\overline{X}) & \lTo & p^{-1}(U) = \{(n, \overline{(2, 3,..., p)}), 1 \leq n \leq p\} \\ 
\dTo^p & & \dTo^{p|_{p^{-1}(U)}}\\
S^p(\overline{X}) & \lTo & U = \{\overline{(1, 2,..., p)}\} \\
\end{diagram}

Over $k$ it looks like this
\begin{diagram}
X \times S^{p-1}(X) & \lTo & p^{-1}(U) \cong Spec(L) \\ 
\dTo^p & & \dTo^{p|_{p^{-1}(U)}}\\
S^p(X) & \lTo & U \cong Spec(k) \\
\end{diagram}

We will use this in theorem \ref{varietytoquadric} and theorem \ref{varietytoSB}.
\end{remark}

\begin{remark} \label{reductiontobirational} Since our problem only concerns birational isomorphism, we can always replace our varieties with birationally isomorphic ones when it suits our purpose but does not change our result. Or we can consider what happens with the generic fiber. For example, in theorem \ref{varietytoquadric} we consider the residue field of the generic fiber of the map $p$ in our construction without mentioning $V, W$ and $X'$.
\end{remark}

\section{When $p = 2, \text{ all } n$ }\label{p=2,alln}
When $p = 2$, we show that the standard norm varieties are birationally isomorphic to Pfister quadrics associated to Pfister forms. This result together with the Chain P-equivalence Theorem and properties of quadratic forms will allow us to compare the standard norm varieties for two equal symbols.

For a quadratic form $\varphi$, let $A_{\varphi}$ denote its symmetric matrix and $D_k(\varphi) \subseteq k$ denote the set of its values. Also for an $n$-tuple $(a_1,..., a_n)$, let $\varphi_n$ denote the $n$-fold Pfister form $\langle\langle a_1,..., a_n \rangle\rangle = \prod_{i = 1}^n \langle 1, - a_i \rangle$. Furthermore we associate to $\varphi_n$ the subform $\psi_n = \langle\langle a_1,..., a_{n-1} \rangle\rangle \perp \langle -a_n \rangle$ and denote the quadric defined by $\psi_n$ as $Z(\psi_n)$, known as a Pfister quadric. Below are a few more definitions. A general reference for quadratic forms is \cite{LamItoQFoverF}.

\dfn Two quadratic forms $\varphi$ and $\varphi'$ are said to be equivalent, written $\varphi \cong \varphi'$, if there exists a matrix $C \in GL(k)$ such that $A_{\varphi'} = CA_{\varphi}C^t$.

\dfn Two Pfister forms $\varphi = \langle\langle a_1,..., a_n \rangle\rangle$ and $\varphi' = \langle\langle a_1',..., a_n' \rangle\rangle$ are said to be simply P-equivalent if there exist two indices $i, j$ such that $\langle\langle a_i, a_j \rangle\rangle \cong \langle\langle a_i', a_j' \rangle\rangle$ and $a_k = a_k'$ for $k \neq i, j$. More generally, they are said to be chain P-equivalent, written $\varphi \approxeq \varphi'$, if there exists a sequence of Pfister forms $\varphi_0, \varphi_1,..., \varphi_{m-1}, \varphi_m$ such that $\varphi = \varphi_0, \varphi' = \varphi_m$ and $\varphi_i$ is simply P-equivalent to $\varphi_{i+1}, 0 \leq i \leq m - 1$.

Clearly $\varphi \approxeq \varphi'$ implies $\varphi \cong \varphi'$. The converse statement was proven by R. Elman and T. Y. Lam in \cite{ELPFandtheKofF} and is called the Chain P-equivalence Theorem. We recall the statement here, for use in proposition \ref{birationalsubforms}.

\begin{theorem}\label{chainP-equivalence}(Chain P-equivalence Theorem) Let $\varphi$ and $\varphi'$ be $n$-fold Pfister forms. Then $\varphi \cong \varphi'$ if and only if $\varphi \approxeq \varphi'$.
\end{theorem}

\dfn Two quadratic forms $\varphi$ and $\varphi'$ are said to be birationally equivalent if the quadrics they define are birationally isomorphic, i.e. if the function fields $k(Z(\varphi))$ and $k(Z(\varphi'))$ are isomorphic.

We begin with a lemma about two equivalent Pfister forms and the matrix that connects them.

\begin{lemma} \label{Cnmatrix} If $\varphi_{n-1}$ and $\varphi_n = \langle 1, -b \rangle \varphi_{n-1}$ are Pfister forms with matrices $A_{\varphi_{n-1}}$ and $A_{\varphi_n}$ and $c = \varphi_n(x_1,..., x_{2^n})$ then $\varphi_n \cong \langle c \rangle \varphi_n$ via a matrix $C_n \in GL_{2^n}(k(x_1,..., x_{2^n}))$, that is $C_nA_{\varphi_n}C_n^t = cA_{\varphi_n}$, which satisfies 2 properties:
\begin{enumerate}
\item $C_n^{-1} = \frac {C_n}{c}$, hence $(C^t_n)^{-1} = \frac {C^t_n}{c}$ as well.
\item first row and first column of $C_n$ are $(x_1\, ... \, x_{2^n})$ and $A_{\varphi_n}(x_1\, ... \, x_{2^n})^t$.
\end{enumerate}
\end{lemma}

\begin{proof} We induce on $n$. For $n = 1$ and $c = x_1^2 - ax_2^2$, we have $\varphi_1 \cong c\varphi_1$ via
$C_1 = \left(\begin{array}{cc} x_1 & x_2 \\ -ax_2 & -x_1 \end{array}\right)$ which satisfies (1) and (2).

Next, write $A_{\varphi_n} = \left(\begin{array}{cc} A_{\varphi_{n-1}} & 0 \\ 0 & -bA_{\varphi_{n-1}}\end{array}\right)$ then $c = \varphi_n(x_1,..., x_{2^n}) = xA_{\varphi_n}x^t = s - bt \in D_k(\varphi_n)$ where $s = \varphi_{n-1}(x_1,..., x_{2^{n-1}})$ and $t = \varphi_{n-1}(x_{2^{n-1} + 1},..., x_{2^n})$ are in $D_k(\varphi_{n-1})$. By induction $\varphi_{n-1} \cong \langle s \rangle \varphi_{n-1}$ via a matrix $C \in GL_{2^{n-1}}(k(x_1,..., x_{2^{n-1}}))$, that is $CA{\varphi_{n-1}}C^t = sA_{\varphi_{n-1}}$, which satisfies 
\begin{enumerate}
\item $C^{-1} = \frac{C}{s}$, hence $(C^t)^{-1} = \frac {C^t}{s}$.
\item first row and first column of $C$ are $(x_1\, ... \, x_{2^{n-1}})$ and $A_{\varphi_{n-1}}(x_1\, ... \, x_{2^{n-1}})^t$.
\end{enumerate}

Similarly, $\varphi_{n-1} \cong \langle t \rangle \varphi_{n-1}$ via $C' \in GL_{2^{n-1}}(F(x_{2^{n-1} + 1},..., x_{2^n}))$ with the same properties. From this, we have
\begin{enumerate}[(i)]
\item\label{casei} $\varphi_n \cong \langle s \rangle \varphi_{n-1} \perp \langle -b \rangle \langle t \rangle \varphi_{n-1} = \langle s, -bt \rangle \varphi_{n-1}$ with
$$\left(\begin{array}{cc} C & 0 \\ 0 & C' \end{array}\right)
\left(\begin{array}{cc} A_{\varphi_{n-1}} & 0 \\ 0 & -bA_{\varphi_{n-1}} \end{array}\right)
\left(\begin{array}{cc} C^t & 0 \\ 0 & C'^t \end{array}\right)
= \left(\begin{array}{cc} sA_{\varphi_{n-1}} & 0 \\ 0 & -btA_{\varphi_{n-1}} \end{array}\right)$$

\item\label{caseii} $\langle s, -bt \rangle \varphi_{n-1} \cong \langle c, -cbst \rangle \varphi_{n-1}$ with
$$\left(\begin{array}{cc} I & I \\ btI & sI \end{array}\right)
\left(\begin{array}{cc} sA_{\varphi_{n-1}} & 0 \\ 0 & -btA_{\varphi_{n-1}} \end{array}\right)
\left(\begin{array}{cc} I & btI \\ I & sI \end{array}\right)
= \left(\begin{array}{cc} cA_{\varphi_{n-1}} & 0 \\ 0 & -cbstA_{\varphi_{n-1}} \end{array}\right)$$

\item\label{caseiii} Let $D = (CC')^{-1} = C'^{-1}C^{-1} = \frac{C'C}{ts}$ then $\langle c, -cbst \rangle \varphi_{n-1} \cong \langle c, -cb \rangle \varphi_{n-1} = \langle c \rangle \varphi_n$ with
\begin{align*}
\left(\begin{array}{cc} I & 0 \\ 0 & D \end{array}\right)
\left(\begin{array}{cc} cA_{\varphi_{n-1}} & 0 \\ 0 & -cbstA_{\varphi_{n-1}} \end{array}\right)
\left(\begin{array}{cc} I & 0 \\ 0 & D^t \end{array}\right)
& = \left(\begin{array}{cc} cA_{\varphi_{n-1}} & 0 \\ 0 & -cbstDA_{\varphi_{n-1}} \end{array}\right)
\left(\begin{array}{cc} I & 0 \\ 0 & D^t \end{array}\right) \\
& = \left(\begin{array}{cc} cA_{\varphi_{n-1}} & 0 \\ 0 & -cbstDA_{\varphi_{n-1}}D^t \end{array}\right) \\
& = \left(\begin{array}{cc} cA_{\varphi_{n-1}} & 0 \\ 0 & -cbst \frac {A_{\varphi_{n-1}}}{st} \end{array}\right) \\
& = \left(\begin{array}{cc} cA_{\varphi_{n-1}} & 0 \\ 0 & -cbA_{\varphi_{n-1}} \end{array}\right)
\end{align*}

\item\label{caseiv} Putting (\ref{casei}), (\ref{caseii}) and (\ref{caseiii}) together, we get $\varphi_n \cong \langle s \rangle \varphi_{n-1} \perp \langle -b \rangle \langle t \rangle \varphi_{n-1} = \langle s, -bt \rangle \varphi_{n-1} \cong \langle c, -cbst \rangle \varphi_{n-1} \cong \langle c, -cb \rangle \varphi_{n-1} = \langle c \rangle \varphi_n$
via $C'_n$ where
\begin{align*}
C'_n
& = \left(\begin{array}{cc} I & 0 \\ 0 & D \end{array}\right)
\left(\begin{array}{cc} I & I \\ btI & sI \end{array}\right)
\left(\begin{array}{cc} C & 0 \\ 0 & C'  \end{array}\right) \\
& = \left(\begin{array}{cc} I & I \\ btD & sD \end{array}\right)
\left(\begin{array}{cc} C & 0 \\ 0 & C' \end{array}\right) \\
& = \left(\begin{array}{cc} C & C' \\ btC'^{-1}C^{-1}C & sC'^{-1}C^{-1}C' \end{array}\right) \\
& = \left(\begin{array}{cc} C & C' \\ bC' & \frac{C'CC'}{t} \end{array}\right)
\end{align*}
\end{enumerate}

At last, let $C_n = \left(\begin{array}{cc} I & 0 \\ 0 & -I \end{array}\right)C'_n = \left(\begin{array}{cc} C & C' \\ -bC' & - \frac{C'CC'}{t} \end{array}\right)$ then its inverse $C_n^{-1} = \frac {C_n}{c}$, its first row and column are $(x_1\, ... \, x_{2^n})$ and $A_{\varphi_n}(x_1\, ... \, x_{2^n})^t$, and
\begin{align*}
C_nA_{\varphi_n}C_n^t
& = \left(\begin{array}{cc} I & 0 \\ 0 & -I \end{array}\right)
C'_nA_{\varphi_n}C_n^{'t}
\left(\begin{array}{cc} I & 0 \\ 0 & -I \end{array}\right) \\
& = cA_{\varphi_n}
\end{align*}

The last equality can be verified directly,
\begin{align*}
C_nA_{\varphi_n}C_n^t
& = \left(\begin{array}{cc} C & C' \\ -bC' & - \frac{C'CC'}{t} \end{array}\right)
\left(\begin{array}{cc} A_{\varphi_{n-1}} & 0 \\ 0 & -bA_{\varphi_{n-1}} \end{array}\right)
\left(\begin{array}{cc} C^t & -bC'^t \\ C'^t & - \frac{C'^tC^tC'^t}{t} \end{array}\right) \\
& = \left(\begin{array}{cc} CA_{\varphi_{n-1}} & -bC'A_{\varphi_{n-1}} \\ -bC'A_{\varphi_{n-1}} & \frac{b}{t}C'CC'A_{\varphi_{n-1}} \end{array}\right)
\left(\begin{array}{cc} C^t & -bC'^t \\ C'^t & - \frac{C'^tC^tC'^t}{t} \end{array}\right) \\
& = \left(\begin{array}{cc} CA_{\varphi_{n-1}}C^t - bC'A_{\varphi_{n-1}}C'^t & -bCA_{\varphi_{n-1}}C'^t + \frac{b}{t}C'A_{\varphi_{n-1}}C'^tC^tC'^t \\ -bC'A_{\varphi_{n-1}}C^t + \frac{b}{t}C'CC'A_{\varphi_{n-1}}C'^t & b^2C'A_{\varphi_{n-1}}C'^t - \frac{b}{t^2}C'CC'A_{\varphi_{n-1}}C'^tC^tC'^t \end{array}\right) \\
& = \left(\begin{array}{cc} sA_{\varphi_{n-1}} - btA_{\varphi_{n-1}} & -bCA_{\varphi_{n-1}}C'^t + bA_{\varphi_{n-1}}C^tC'^t \\ -bC'A_{\varphi_{n-1}}C^t + bC'CA_{\varphi_{n-1}} & b^2tA_{\varphi_{n-1}} - bsA_{\varphi_{n-1}} \end{array}\right) \\
& = \left(\begin{array}{cc} cA_{\varphi_{n-1}} & -bCA_{\varphi_{n-1}}C'^t + \frac{b}{s}CA_{\varphi_{n-1}}C^tC^tC'^t \\ -bC'A_{\varphi_{n-1}}C^t + \frac{b}{s}C'CCA_{\varphi_{n-1}}C^t & -bcA_{\varphi_{n-1}} \end{array}\right) \\
& = \left(\begin{array}{cc} cA_{\varphi_{n-1}} & -bCA_{\varphi_{n-1}}C'^t + bCA_{\varphi_{n-1}}C'^t \\ -bC'A_{\varphi_{n-1}}C^t + bC'A_{\varphi_{n-1}}C^t & -bcA_{\varphi_{n-1}} \end{array}\right) \\
& = \left(\begin{array}{cc} cA_{\varphi_{n-1}} & 0 \\ 0 & -bcA_{\varphi_{n-1}}\end{array}\right) \\
& = cA_{\varphi_n}
\end{align*}
\end{proof}

This next lemma is needed to show that the residue field in theorem \ref{varietytoquadric} stays the same.

\begin{lemma} \label{subCnmatrix} The $n \times n$ matrix
$$M= \left(\begin{array}{cccc}
a_1b_1 & a_1b_2 & . \, . & a_1b_n \\
a_2b_1 & a_2b_2 & . \, . & a_2b_n \\
: & : & . & : \\
a_nb_1 & a_nb_2 & . \, . & a_nb_n \end{array}\right)$$
has characteristic polynomial $\text{char}_M(t) = t^{n-1}(t - a_1b_1 - a_2b_2 - ... - a_nb_n)$.
\end{lemma}

\begin{proof} We consider what $M$ does to the standard basis,
\begin{align*}
k^n & \rTo^M k^n \\
(1, 0,..., 0) & \mapsto b_1(a_1,..., a_n) \\
(0, 1,..., 0) & \mapsto b_2(a_1,..., a_n) \\
& : \\
(0, 0,..., 1) & \mapsto b_n(a_1,..., a_n)
\end{align*}

Thus $M$ sends the vector $(a_1,..., a_n)$ to $\alpha(a_1,..., a_n)$ where $\alpha = a_1b_1 + a_2b_2 + ... + a_nb_n$. Letting $v_1 = (a_1,..., a_n)$, we choose a new basis $\{v_1,..., v_n\}$ for $k^n$ such that $ker(M) = \langle v_2,..., v_n \rangle$ and again look at what M does as a linear map,
\begin{align*}
k^n & \rTo^M k^n \\
v_1 & \mapsto (\alpha, 0,..., 0) \\
v_2 & \mapsto (0,..., 0) \\
& : \\
v_n & \mapsto (0,..., 0)
\end{align*}

Therefore M has canonical form
$$\left(\begin{array}{cccc} \alpha & 0 & . \, . & 0 \\
0 & 0 & . \, . & 0 \\
: & : & . & : \\
0 & 0 & . \, . & 0 \end{array}\right)$$
and $\det(tI - M) = \text{char}_M(t) = t^n - \alpha t^{n-1} = t^{n-1}(t - a_1b_1 - a_2b_2 - ... - a_nb_n)$.
\end{proof}

We are now ready to turn the standard norm varieties into Pfister quadrics defined by subforms of Pfister forms.

\begin{theorem} \label{varietytoquadric} The standard norm variety $X(a_1,..., a_n)$ for $\{a_1,..., a_n\}$ is birationally isomorphic to the Pfister quadric $Z(\psi_n) \subset \mathbb{P}^{2^{n-1}}_k$ defined by the subform $\psi_n = \langle\langle a_1,..., a_{n-1} \rangle\rangle \perp \langle -a_n \rangle$ of the Pfister form $\varphi_n = \langle\langle a_1,..., a_n \rangle\rangle$.
\end{theorem}

\begin{proof} We induce on $n$. First we verify the case $n = 2$. As described in remark \ref{startofconstruction}, we begin our symmetric power construction with $X(a_1) = Spec(L)$ where $L = k(\sqrt{a_1})$ and get
\begin{diagram}
Spec(L) \times Spec(L) & \lTo & p^{-1}(U) \cong Spec(L) \\
\dTo^p & & \dTo^{p|_{p^{-1}(U)}} \\
S^2(Spec(L)) & \lTo & U \cong Spec(k) \\
\end{diagram}

Hence $X(a_1, a_2) = Z(N_{L/k } - a_2) = Z(x_1^2 - a_1x_2^2 - a_2)$ the hypersurface defined by the equation $N_{L/k } - a_2 = x_1^2 - a_1x_2^2 - a_2 = 0$. Projectivization then gives $X(a_1, a_2) = Z(x_1^2 - a_1x_2^2 - a_2x_3^2) = Z(\psi_2) \subset \mathbb{P}^2_k$ as required.

By induction $X(a_1,..., a_{n+1}) \approx Z(\psi_{n+1})$. Write $\psi = \psi_{n+1} = \varphi_n \perp \langle -a_{n+1} \rangle = \langle 1 \rangle \perp \varphi' \perp \langle -a_{n+1} \rangle \cong \langle 1, -a_{n+1} \rangle \perp \varphi'$ where $\varphi'$ is the pure subform of $\varphi$. From construction we get
$$(X_{n+1} \times X_{n+1}) \setminus \bigtriangleup \longrightarrow ((X_{n+1} \times X_{n+1}) \setminus \bigtriangleup)/ _{S_2} \longrightarrow Gr(2, \mathbb{A}^{2^n + 1}_k)$$

Let $U = \langle u, v \rangle = \langle (1, 0, x_2,..., x_{2^n}), (0, 1, y_2,..., y_{2^n}) \rangle$ be the generic plane in $\mathbb{A}^{2^n + 1}_k$ and moreover let $\{u, v\}$ be a basis for $U$. Over this basis, the restriction $\psi_{k(x_i, y_i)}|_U$ has matrix form
$$\left(\begin{array}{cc} \psi(u) & b(u, v) \\ b(u, v) & \psi(v) \end{array}\right)$$
where
\begin{align*}
U \times U & \rTo^{b} k \\
(u', v') & \mapsto \frac{1}{2}(\psi(u' + v') - \psi(u') - \psi(v'))
\end{align*}
is the symmetric bilinear form associated to $\psi_{k(x_i, y_i)}|_U$.

The generic fiber is then the point $(r, s) \in U$ such that
\begin{align*}
\psi(r, s) & = \psi(u, u)r^2 + 2b(u, v)rs + \psi(v, v)s^2 \\
& = 0
\end{align*}
with residue field
$$qf \left(k(x_i, y_i)\left[\frac{r}{s}\right] / \left(\psi(u, u)\left(\frac{r}{s}\right)^2 + 2b(u, v)\frac{r}{s} + \psi(v, v)\right)\right) = k(x_i, y_j)(\sqrt{-\theta})$$
where
\begin{align*}
\theta & = \psi(u)\psi(v) - b(u, v)^2 \\
& = (1+ \varphi'(x_2,..., x_{2^n}))(-a_{n+1} + \varphi'(y_2,..., y_{2^n})) - b(u, v)^2 \\
& = (\varphi(1, x_2,..., x_{2^n}))(-a_{n+1} + \varphi'(y_2,..., y_{2^n})) - b(u, v)^2 \\
& = (-a_{n+1})\varphi(1, x_2,..., x_{2^n}) + \varphi(1, x_2,..., x_{2^n})\varphi(0, y_2,..., y_{2^n}) - b(u, v)^2
\end{align*}

If we write $\varphi = \langle 1, c_2,..., c_{2^n} \rangle $ then by lemma \ref{Cnmatrix}, there exists a matrix
$$C_n = \left(\begin{array}{cccc}
1 & x_2 & . \, . & x_{2^n} \\
c_2x_2 & . & . \, . & . \\
 : & : & . & : \\
c_{2^n}x_{2^n} & . & . \, . & . \end{array}\right)$$
such that $\varphi(1, x_2,..., x_{2^n})\varphi(0, y_2,..., y_{2^n}) = \varphi((0, y_2,..., y_{2^n})C_n)$. So
\begin{align*}
\theta & = (-a_{n+1})\varphi(1, x_2,..., x_{2^n}) + \varphi((0, y_2,..., y_{2^n})C_n) - b(u, v)^2 \\
& = (-a_{n+1})\varphi(1, x_2,..., x_{2^n}) + \varphi((0, y_2,..., y_{2^n})A_{\varphi}(1, x_2,..., x_{2^n})^t, z_2,..., z_{2^n}) \\
&\hspace{0.5cm} - ((y_2,..., y_{2^n})A_{\varphi'}(x_2,..., x_{2^n})^t)^2 \\
& = (-a_{n+1})\varphi(1, x_2,..., x_{2^n}) + \varphi((y_2,..., y_{2^n})A_{\varphi'}(x_2,..., x_{2^n})^t, z_2,..., z_{2^n}) \\
&\hspace{0.5cm} - ((y_2,..., y_{2^n})A_{\varphi'}(x_2,..., x_{2^n})^t)^2 \\
& = (-a_{n+1})\varphi(1, x_2,..., x_{2^n}) + \varphi'(z_2,..., z_{2^n})
\end{align*}

Above, we let $(z_1, z_2,..., z_{2^n}) = (0, y_2,..., y_{2^n})C_n$, which means $(z_2,..., z_{2^n}) = (y_2,..., y_{2^n})M$ where $M$ is $C_n$ without its first row and column. Since $C_n^2 = \varphi(1, x_2,..., x_{2^n})I$, it follows $M^2 = \varphi(1, x_2,..., x_n)I - (c_ix_ix_j), 2 \leq i, j \leq 2^n$. By lemma \ref{subCnmatrix}, $\det(M^2) = \varphi(1, x_2,..., x_{2^n})^{2^n - 2}$. Thus $\det(M) = \varphi(1, x_2,..., x_{2^n})^{2^{n-1} - 1}$ and $M \in GL_{2^n - 1}(F(x_2,..., x_{2^n}))$. So the residue field stays the same
$$F(x_i, y_j)(\sqrt{-\theta}) = F(x_i, z_j)(\sqrt{-\theta})$$

It has quadratic norm
\begin{align*}
N(m+n \sqrt{-\theta}) & = m^2 - a_{n+1}\varphi(1, x_2,..., x_{2^n})n^2 + \varphi'(z_2,..., z_{2^n})n^2 \\
& = \varphi (m, nz_2,..., nz_{2^n}) - a_{n+1}\varphi(n, nx_2,..., nx_{2^n}) \\
& = \langle 1, -a_{n+1} \rangle \varphi (m, nz_2,..., nz_{2^n}, n, nx_2,..., nx_{2^n}) \\
& = \varphi_{n+1}(t_1,..., t_{2^{n+1}})
\end{align*}

Therefore our projectivized $X(a_1,..., a_{n+2}) = Z(N - a_{n+2}t_{2^{n+1} + 1}^2)$ is birationally isomorphic to $Z(\varphi_{n+1} \perp \langle -a_{n+2} \rangle) = Z(\psi_{n+2}) \subset \mathbb{P}^{2^{n+1}}_k$ as wanted.
\end{proof}

Next, we show that interchanging $a_i$ and $a_j$ or multiplying $a_i$ by any nonzero norm $N_{k(\sqrt{a_j})/k}(u)$ in the symbol $\{a_1,..., a_n\}$ does not change its standard norm variety. For this, we need two more lemmas about Pfister neighbors, the first one we will use toward our corollary \ref{p=2corollary} and the second one we will use toward our example \ref{p=2example}.

\begin{lemma} \label{interchange4forms} If $\varphi = \langle\langle a_1,..., a_n \rangle\rangle$ is an anisotropic Pfister form then the two forms $\varphi \perp \langle -b\varphi \rangle \perp \langle - c \rangle$ and $\varphi \perp \langle - c\varphi \rangle \perp \langle -b \rangle$ are birationally equivalent.
\end{lemma}

\begin{proof} We connect the quadrics defined by these two forms by a sequence of birationally isomorphic ones. Let $(x, y, z)$ be the generic zero for the form $\varphi \perp \langle -b\varphi \rangle \perp \langle - c \rangle$ then
$$\varphi(x) - b\varphi(y) - cz^2 = 0$$

Since $\varphi$ is Pfister and $\varphi(y) \in D_{k(y)}(\varphi)$, it follows $\varphi \cong \varphi(y)\varphi$ over $k(y)$. That means there exists a matrix $C \in GL(k(y))$ such that $\varphi(x) = \varphi(y)\varphi(Cx)$. Let $x' = Cx$ then $ k(x, y, z) = k(x', y, z)$ and
\begin{align*}
\varphi(y)\varphi(x') - b\varphi(y) - cz^2 & = 0 \text{, hence} \\
\varphi(x') - b - c\frac{z^2}{\varphi(y)} & = 0
\end{align*}

Now let $y' = \frac{y}{\varphi(y)}$ then $ k(x, y, z) = k(x', y', z)$ and
\begin{align*}
\varphi(x') - b - cz^2\varphi(y') & = 0 \text{, hence} \\
\frac{\varphi(x')}{z^2} - \frac{b}{z^2} - c\varphi(y') & = 0
\end{align*}

At last let $x'' = \frac{x'}{z}$ and $z' = \frac{1}{z}$ then $(x'', y', z')$ is a generic zero for $\varphi \perp \langle - c\varphi \rangle \perp \langle -b \rangle, k(x, y, z) = k(x'', y', z')$ and
$$\varphi(x'') - c\varphi(y') - bz'^2 = 0$$

Therefore the two forms $\varphi \perp \langle -b\varphi \rangle \perp \langle - c \rangle$ and $\varphi \perp \langle - c\varphi \rangle \perp \langle -b \rangle$ are birationally equivalent.
\end{proof}

\begin{lemma} \label{scalar4forms} If $\varphi = \langle\langle a_1,..., a_n \rangle\rangle$ is a anisotropic Pfister form then the two forms $\varphi \perp \langle -b \rangle$ and $\varphi \perp \langle -b\varphi(x_0) \rangle$ with $\varphi(x_0) \neq 0$ are birationally equivalent. In particular $\varphi \perp \langle -b \rangle \approx \varphi \perp \langle -bN_{k(\sqrt{a_i})/k}(u) \rangle$ for any nonzero norm $N_{k(\sqrt{a_i})/k}(u)$.
\end{lemma}

\begin{proof} We use the same approach as in lemma \ref{interchange4forms}. Let $(x, y)$ be a generic zero for the form $\varphi \perp \langle -b\varphi(x_0) \rangle$ then
\begin{align*}
\varphi(x) - b \varphi(x_0)y^2
& = 0 \text{, hence} \\
\varphi(x_0)\varphi(x) - b\varphi(x_0)^2y^2
& = 0
\end{align*}

Again $\varphi \cong \varphi(x_0)\varphi$ over $k$, i.e. there exists a matrix $C \in GL(k)$ such that $\varphi(Cx) = \varphi(x_0)\varphi(x)$. Let $x' = Cx$ and $y' = \varphi(x_0)y$ then $(x', y')$ is a generic zero for $\varphi \perp \langle -b \rangle$, $k(x, y) = k(x', y')$ and
$$\varphi(x') - by'^2 = 0$$

Therefore the two forms $\varphi \perp \langle -b \rangle$ and $\varphi \perp \langle -b\varphi(x_0) \rangle$ with $\varphi(x_0) \neq 0$ are birationally equivalent. The last statement follows when we choose $x_0$ such that $\varphi(x_0) = N_{k(\sqrt{a_i})/k}(u)$.
\end{proof}

\begin{proposition} \label{birationalsubforms} If two Pfister form $\varphi$ and $\varphi'$ are equivalent then their associated subforms $\psi$ and $\psi'$ are birationally equivalent.
\end{proposition}

\begin{proof} By the Chain P-equivalence Theorem, $\varphi \approxeq \varphi'$. So there exists a sequence of Pfister forms $\varphi_0, \varphi_1,..., \varphi_t,..., \varphi_{m-1}, \varphi_m$ such that $\varphi = \varphi_0, \varphi' = \varphi_m$ and $\varphi_t$ is simply P-equivalent to $\varphi_{t+1}, 0 \leq t \leq m - 1$. Write $\varphi_t = \langle\langle a_1,..., a_i,..., a_j,..., a_n \rangle\rangle$ and $\varphi_{t+1} = \langle\langle a_1,..., a_i',..., a_j',..., a_n \rangle\rangle$ where $\langle\langle a_i, a_j \rangle\rangle \cong \langle\langle a_i', a_j' \rangle\rangle$. If $i = j$ then there is nothing to do. Otherwise, we consider each case separately.
\begin{enumerate}
\item If $j \neq n$ then
\begin{align*}
\psi_t & = \langle\langle a_1,..., a_i,..., a_j,..., a_{n-1} \rangle\rangle \perp \langle -a_n \rangle \\
& \cong \langle\langle a_1,..., a_i',..., a_j',..., a_{n-1} \rangle\rangle \perp \langle -a_n \rangle \\
& = \psi_{t+1}
\end{align*}

\item If $j = n$ and $i \neq n - 1$ then by lemma \ref{interchange4forms},
\begin{align*}
\psi_t & = \langle\langle a_1,..., a_i,..., a_{n-1} \rangle\rangle \perp \langle -a_j \rangle \\
& \approx \langle\langle a_1,..., a_i,..., a_j \rangle\rangle \perp \langle -a_{n-1} \rangle \\
& \cong \langle\langle a_1,..., a_i',..., a_j' \rangle\rangle \perp \langle -a_{n-1} \rangle \\
& \approx \langle\langle a_1,..., a_i',..., a_{n-1} \rangle\rangle \perp \langle -a_j' \rangle \\
& = \psi_{t+1}
\end{align*}

\item If $j = n$ and $i = n - 1$ then again by lemma \ref{interchange4forms},
\begin{align*}
\psi_t & = \langle\langle a_1,..., a_{n-2}, a_i \rangle\rangle \perp \langle -a_j \rangle \\
& \cong \langle\langle a_1,..., a_i, a_{n-2} \rangle\rangle \perp \langle -a_j \rangle \\
& \approx \langle\langle a_1,..., a_i, a_j \rangle\rangle \perp \langle -a_{n-2} \rangle \\
& \cong \langle\langle a_1,..., a_i', a_j' \rangle\rangle \perp \langle -a_{n-2} \rangle \\
& \approx \langle\langle a_1,..., a_i', a_{n-2} \rangle\rangle \perp \langle -a_j' \rangle \\
& \cong \langle\langle a_1,..., a_{n-2}, a_i' \rangle\rangle \perp \langle -a_j' \rangle \\
& = \psi_{t+1}
\end{align*}
\end{enumerate}

Hence $\psi_t \approx \psi_{t+1}$ for all $t$, and $\psi \approx \psi'$.
\end{proof}

\begin{remark} \label{specialPneighbors} Let $\varphi$ be a Pfister form of dimension greater than or equal to 2, $c \in k^{\times}$, and $\varphi_1$ a nonzero subform of $\varphi$. In \cite{AhmadFFofPN} H. Ahmad called $(\varphi, c, \varphi_1) $ a Pfister triple, $\varphi \perp \langle c \rangle$ the base form, $\varphi \perp c\varphi_1$ the form defined by the triple, $\varphi \perp c\varphi$ the associated Pfister form, and any form similar to such $\varphi \perp c\varphi_1$ a special Pfister neighbor. In this setting the forms in lemma \ref{interchange4forms} and the forms in lemma \ref{scalar4forms} are pairwise special Pfister neighbors of the same dimensions and have the same associated Pfister forms $\varphi \otimes \langle\langle b, c \rangle\rangle$ and $\varphi \otimes \langle\langle b \rangle\rangle$, respectively. The lemmas then follow from his more general \cite[1.6]{AhmadFFofPN}.
\end{remark}

\begin{remark} \label{stronglymultiplicativeforms} One sees that lemmas \ref{interchange4forms} and \ref{scalar4forms} hold for any strongly multiplicative form $\varphi$ as defined in \cite{LamItoQFoverF}. The work lies with anisotropic Pfister forms. The remaining strongly multiplicative forms are isotropic, hence their function fields are rational and both lemmas become trivial.
\end{remark}

Proposition \ref{birationalsubforms} enables us to compare the standard norm varieties for two equal symbols.

\begin{corollary} \label{p=2corollary} The standard norm varieties $X(a_1,..., a_n)$ and $X(b_1,..., b_n)$ for $\{a_1,..., a_n\}$ and $\{b_1,..., b_n\}$ are birationally isomorphic if $\{a_1,..., a_n\} = \{b_1,..., b_n\}$ in $K^M_n(k)/2$.
\end{corollary}

\begin{proof} By \cite[6.20]{EKMAandGTofQF}, the two Pfister forms $\varphi = \langle\langle a_1,..., a_n \rangle\rangle$ and $\varphi' = \langle\langle b_1,..., b_n \rangle\rangle$ are equivalent. Proposition \ref{birationalsubforms} now implies their associated subforms $\psi$ and $\psi'$ are birationally equivalent. By theorem \ref{varietytoquadric}, $X(a_1,..., a_n)$ and $X(b_1,..., b_n)$ are birationally isomorphic.
\end{proof}

\begin{example} \label{p=2example} For any nonzero norm $N_{k(\sqrt{a_i})/k}(u)$, we know $\{a_1,..., a_i,..., a_j,..., a_n\} = \{a_1,..., a_i,..., a_jN_{k(\sqrt{a_i})/k}(u),..., a_n\}$ in $K^M_n(k)/2$. By corollary \ref{p=2corollary}, their standard norm varieties are birationally isomorphic. Or we can use theorem \ref{varietytoquadric} and lemma \ref{scalar4forms}, bypassing the Chain P-equivalence Theorem to see this as well.
\end{example}

\section{When $p > 2 \text{ and } n = 2$}\label{p>2,n=2}
When $p > 2, n = 2$, we show that the standard norm varieties are birationally isomorphic to Severi-Brauer varieties.

\begin{theorem} \label{varietytoSB} The standard norm variety $X(a, b)$ for $\{a, b\}$ is birationally isomorphic to the Severi-Brauer variety $SB(A)$ associated to the cyclic algebra $A = (a, b, \zeta_p)_k$.
\end{theorem}

\begin{proof} Again if we start the symmetric power construction with $X(a) = Spec (L)$ where $L = k(\sqrt[p]{a})$ then $X(a, b) = Z(N_{L/k} - b)$ by remark \ref{startofconstruction}. We consider what happens in a split case, where $A_L \cong M_p(L)$ and $SB(A_L) \cong \mathbb{P}^{p-1}_L$. Furthermore, if $G = Gal(L/k) = \langle \sigma \rangle$ of order $p$ then over $L$, the norm $N_{L/k}(x)$ splits in to a product $\prod_{i=0}^{p-1}\sigma^i(x)$ for every $x \in L$. Define
$$U_L = \left\{I \subset M_p(L) \text{ where } I = \left\{
\left(\begin{array}{cccc}
\alpha_0 & 0 & . \, . & 0 \\ 
. & . & . \, . & . \\
: & : & . & : \\
\alpha_{p-1} & 0 & . \, . & 0
\end{array}\right) M,
\alpha_i \neq 0 \text{ for all } i \text{ and } M \in M_p(L) \right\} \right\}$$
then $U_L$ is an open subset in $SB(A_L)$ and we have a diagram
\begin{diagram}
Z(N_{L/k} - b)_L & \rTo^{f_L} & U_L & \rTo^{open} & SB(A_L) \\
\dTo_{/G} & & \dTo_{/G} & & \dTo_{/G} \\
Z(N_{L/k} - b) & \rTo^f & U & \rTo^{open} & SB(A) \\
\end{diagram}
where $f_L$ can be described as follows:
\begin{align*}
Z(N_{L/k} - b)_L & \rTo^{f_L} U_L \\
(x, \sigma(x),..., \sigma^{p-1}(x)) & \mapsto (x : x\sigma(x) :...: x\sigma(x) ... \sigma^{p-1}(x))
\end{align*}
if we abuse notation and write points in $SB(A_L)$ in projective coordinates. We verify that $f_L$ is $G$-equivariant,
\begin{align*}
f_L(\sigma \cdot (x, \sigma(x), ... , \sigma^{p-2}(x), \sigma^{p-1}(x))) & = f(\sigma(x), \sigma^2(x), ... , \sigma^{p-1}(x), \sigma^p(x)) \\
& = (\sigma(x) : \sigma(x)\sigma^2(x) : ... : \sigma(x) ... \sigma^{p-1}(x) : \sigma(x)\sigma^2(x) ... \sigma^p(x)) \\
& = (\sigma(x) : \sigma(x)\sigma^2(x) : ... : \sigma(x) ... \sigma^{p-1}(x) : b)
\end{align*}
while
\begin{align*}
\sigma \cdot f_L(x, \sigma(x),..., \sigma^{p-1}(x)) & = \left(\begin{array}{ccccc} 0 & 1 & 0 & . & . \\ 0 & 0 & 1 & . & . \\ : & : & : & : & : \\ 0 & . & . & 0 & 1 \\ b & 0 & . & . & 0 \end{array}\right) \left(\begin{array}{c} x \\ x\sigma(x) \\ : \\ x\sigma(x) ... \sigma^{p-2}(x) \\ x\sigma(x) ... \sigma^{p-1}(x)\end{array}\right) \\
& = (x\sigma(x) : x\sigma(x)\sigma^2(x) : ... : x\sigma(x) ... \sigma^{p-1}(x) : bx) \\
& = (\sigma(x) : \sigma(x)\sigma^2(x) : ... : \sigma(x) ... \sigma^{p-1}(x): b)
\end{align*}

In function fields, we have an isomorphism of the same name $f_L$ from $L(U_L) = L(\frac{t_1}{t_0}, ..., \frac{t_p}{t_0})$ to $L(Z(N_{L/k} - b)_L) = L(x, \sigma(x),..., \sigma^{p-1}(x))$,
\begin{align*}
L(\frac{t_1}{t_0}, ... , \frac{t_p}{t_0})
& \rTo^{f_L} L(x, \sigma(x), ... , \sigma^{p-1}(x)) \\
\frac{t_i}{t_0}
& \mapsto x\sigma(x) ... \sigma^{i-1}(x)
\end{align*}
where $i = 1, ... , p$ and $\frac{t_p}{t_0} = b$ with inverse
\begin{align*}
L(x, \sigma(x), ... , \sigma^{p-1}(x)) & \rTo^{f_L^{-1}} L(\frac{t_1}{t_0}, ... , \frac{t_p}{t_0}) \\
\sigma^{i-1}(x) & \mapsto \frac{t_i}{t_{i-1}}
\end{align*}

We verify that $f_L$ respects the $G$-action,
\begin{align*}
f_L(\sigma \cdot \frac{t_i}{t_0})
& = f_L(\frac{t_{i+1}}{t_1}) \\
& = f_L((\frac{t_{i+1}}{t_0})(\frac{t_1}{t_0})^{-1}) \\
& = x\sigma(x) ... \sigma^i(x)x^{-1} \\
& = \sigma(x) ... \sigma^i(x)
\end{align*}
while
\begin{align*}
\sigma \cdot f_L(\frac{t_i}{t_0}) & = \sigma \cdot (x\sigma(x) ... \sigma^{i-1}(x)) \\
& = \sigma(x) ... \sigma^i(x)
\end{align*}

Therefore $Z(N_{L/k} - b)_L$ is birationally isomorphic to $U_L$. So $Z(N_{L/k} - b)$ is birationally isomorphic to $U$, hence to $SB(A)$.
\end{proof}

This theorem enables us to compare the standard norm varieties for two equal symbols.

\begin{corollary} \label{n=2corollary} The standard norm varieties $X(a_1, a_2)$ and $X(b_1, b_2)$ for $\{a_1, a_2\}$ and $\{b_1, b_2\}$ are birationally isomorphic if $\{a_1, a_2\} = \{b_1, b_2\}$ in $K^M_2(k)/p$.
\end{corollary}

\begin{proof}
By the norm residue homomorphism $K^M_2(k)/p \rTo Br_p(k)$, the classes of $(a_1, a_2, \zeta_p)_k$ and $(b_1, b_2, \zeta_p)_k$ are equal in the subgroup $Br_p(k)$ of elements of exponent $p$ in the Brauer group $Br(k)$. Since they have the same dimension, $(a_1, a_2, \zeta_p)_k$ and $(b_1, b_2, \zeta_p)_k$ are isomorphic as algebras. Hence $SB((a_1, a_2, \zeta_p)_k) \cong SB((a_1, a_2, \zeta_p)_k)$. It follows from the theorem that $X(a_1, a_2) \approx X(b_1, b_2)$.
\end{proof}

\section{When $p> 2 \text{ and } n = 3$}\label{p>2,n=3}
When $p > 2, n = 3$, we show that the standard norm varieties are birationally isomorphic to varieties defined by reduced norms of cyclic algebras.

\begin{theorem} \label{varietytovarietydefinedbyreducednormofcyclicalgebra} The standard norm variety $X(a, b, c)$ for $\{a, b, c\}$ is birationally isomorphic to $Z(Nrd_{A/k} - c)$, where $A = (a, b, \zeta_p)_k$.
\end{theorem}

\begin{proof} We consider what happens in a split case. Let $L = k(\sqrt[p]{a})$ and use $SB(A)$ as the standard norm variety $X(a, b)$ for $\{a, b\}$. Again $A_L \cong M_p(L)$ and $SB(A_L) \cong \mathbb{P}^{p-1}_L$. Our symmetric power construction looks like the front square over $k$ and the back square over $L$,
\begin{diagram}
 & & SB(A_L) \times S^{p-1}(SB(A_L)) & \lTo & & & p^{-1}(U_L) & & \\
 & \ldTo^{/G} & \vLine & & & \ldTo^{/G} & & & \\
SB(A) \times SB^{p-1}(A) & & & \lTo & p^{-1}(U) & & & & \\
 \dTo^p & & \dTo^{p_L} & & \dTo^{p|} & & \dTo^{p_L|} & & \\
 & & S^p(SB(A_L)) & \lTo & \VonH & \hLine & U_L & \lTo^{\pi_L} & V_L \\
 & \ldTo^{/G} & & & & \ldTo^{/G} & & \ldTo^{/G} \\
S^p(SB(A)) & & & \lTo & U & \lTo^{\pi} & V \\
\end{diagram}

Now let $X_L$ denote the variety of all \'{e}tale subalgebras of degree $p$ in $End_L(L^p)$. If each subalgebra $D \in X_L$ is generated by a matrix $\lambda$ where $\lambda = (\lambda_1,..., \lambda_p)$ its diagonal form then $S_p$ acts trivially on $X_L$ by permuting the diagonal entries. So we have an $S_p$-equivariant map
\begin{align*}
U_L & \rTo^{f_L} X_L \\
(u_1,..., u_p) & \mapsto D
\end{align*}
where $D$ is the \'{e}tale subalgebra whose eigenspaces are the lines $u_1,..., u_p$, with inverse $f_L^{-1} : D \mapsto (u_1,..., u_p)$. This map fits into the following commutative diagram,
\begin{diagram}
U_L & \rTo^{f_L} & X_L \\
\dTo_{/G} & & \dTo_{/G} \\
U & \rTo^f & X \\
\end{diagram}
and we get vector bundles over the last diagram,
\begin{diagram}
 & & U_L & & & \lTo^{\pi_L} & V_L \\
 & \ldTo^{/G} & \dTo_{f_L} & & & \ldTo & \dTo_{f_L^*} \\
U & & & \lTo^{\pi} & V & & \\
\dTo_f & & & & \dTo_{f^*} & & \\
 & & X_L & & & \lTo^{\pi_{X_L}} & V_{X_L} \\
 & \ldTo^{/G} & & & & \ldTo & \\
X & & & \lTo^{\pi_{X}} & V_X & & \\
\end{diagram}

For each $(u_1,..., u_p) \in U_L$, the preimage $p_L^{-1}((u_1,..., u_p))$ consists of $p$ points $y_1,..., y_p$ where each $y_i$ is of the form $(u_i, (u_1,..., \check{u}_i,..., u_p))$. So $\pi_L^{-1}((u_1,..., u_p)) = \{((u_1,..., u_p), x_1,..., x_p) \mid x_i \in L(y_i)\}$. Correspondingly, $\pi_{X_L}^{-1}(D) = \{(D, d) \mid d \in D\}$. Both are algebras of rank $p$ over $L$.
We can describe the back face of the cube pointwise,
\begin{diagram}
((u_1,..., u_p), x_1,..., x_p) & \rTo^{f_L^*} & \left(D, \left(\begin{array}{cccc} x_1 & . & . & . \\ . & . & 0 & . \\ . & 0 & . & . \\ . & . & . & x_p \end{array}\right) \right) \\
\dTo^{\pi_L} & & \dTo^{\pi_{X_L}} \\
(u_1,..., u_p) & \rTo^{f_L} & D \\
\end{diagram}

Note that if $q(t) = a_1t + ... + a_pt^p$ and $D \ni d = q(\lambda)$ with eigenvalues $q(\lambda_i)$ then $f_L^{*^{-1}}(D, d) = ((u_1,..., u_p), q(\lambda_1),..., q(\lambda_p))$.

Therefore in $V_L$ and $V_{X_L}$ we have two birationally isomorphic subvarieties $Z(N - c)_L$ and $Z(Nrd_{A_L/L} - c)$, since
\begin{align*}
Z(N - c)_L & = \{((u_1,..., u_p), x_1,..., x_p) \mid x_1... x_p = c\} \\
& \cong \{(D, d) \mid D \subset A_L \text{ \'{e}tale of rank } p \text{ and } d \in D \text{ with } N_{D/L}(d) = c\} \\
& = \{(D, d) \mid D \subset A_L \text{ \'{e}tale of rank } p \text{ and } d \in D \text{ with } Nrd_{A_L/L}(d) = c\} \\
& \cong \{d \in A_L \mid \langle d \rangle \subset A_L \text{ \'{e}tale of rank } p \text{ and } Nrd_{A_L/L}(d) = c\} \text{ (via } (D, d) \mapsto d) \\
& = \{d \in A_L \mid Nrd_{A_L/L}(d) = c\} \cap \{d \in A_L \mid \text{its minimal polynomial } m_d(t) \\
& \hspace{0.5cm} \text{ is of degree } p\} \\
& = \{d \in A_L \mid Nrd_{A_L/L}(d) = c\} \cap \{d \in A_L \mid x_i \neq x_j \text{ for all of its} \\
& \hspace{0.5cm} \text{ eigenvalues } x_i, x_j\} \\
& \approx \{d \in A_L \mid Nrd_{A_L/L}(d) = c\} \\
& = Z(Nrd_{A_L/L} - c)
\end{align*}

Note that the intersection above is nonempty, it contains for example the diagonal matrix $(\frac{c}{\zeta_p^{(p - 1)/2}}, \zeta_p,..., \zeta_p^{p-1})$, and the second set is open. Hence our standard norm variety $X(a, b, c) = Z(N - c)$ is birationally isomorphic to $Z(Nrd_{A/k} - c)$ over $k$.\
\end{proof}

Knowing that $X(a, b, c)$ is birationally isomorphic to $Z(Nrd_{A/k} - c)$, where $A = (a, b, \zeta_p)_k$ may allow us to compare $X(a, b, c)$ and $X(a', b', c')$ when $\{a, b, c\} = \{a', b', c'\}$ in $K^M_3(k)/p$. If we know $Z(Nrd_{A/k} - c) \approx Z(Nrd_{A'/k} - c')$, where $A' = (a', b', \zeta_p)_k$ then we can draw the same corollary for $p > 2, n = 3$ as we did for $p = 2$ in \ref{p=2corollary} and for $p > 2, n = 2$ in \ref{n=2corollary}.

\end{document}